\documentclass{amsart}
\usepackage{amsmath,amssymb,amsthm,hyperref,graphicx}
\usepackage{latexsym,mathtools,amsthm}
\usepackage{framed}

\newtheorem{thm}{Theorem}[section]
\newtheorem{prop}[thm]{Proposition}
\newtheorem{cor}[thm]{Corollary}
\theoremstyle{definition}
\newtheorem{defn}[thm]{Definition}

\newtheorem{ex}[thm]{Example}

\newtheorem{conj}[thm]{Conjecture}

\newtheorem{q}[thm]{Question}
\newtheorem{rem}[thm]{Remark}
\newtheorem{lemma}[thm]{Lemma}

\def\height{\operatorname{ht}}

\def\ZZ{{\mathbb Z}}

\newcommand{\field}{\Bbbk}

\newcommand{\maxideal}{\mathfrak{m}}
\newcommand{\nbar}{\mathbf{n}}
\newcommand{\xbar}{\mathbf{x}}
\newcommand{\abar}{\mathbf{a}}
\newcommand{\bbar}{\mathbf{b}}

\newcommand{\vbar}{\mathbf{v}}
\newcommand{\ebar}{\mathbf{e}}
\newcommand{\rbar}{\mathbf{r}}
\newcommand{\sbar}{\mathbf{s}}
\newcommand{\onebar}{\mathbf{1}}

\newcommand{\p}{P}
\newcommand{\poly}{\mathcal{P}}
\newcommand{\symbpoly}{\mathcal{Q}}
\newcommand{\lattice}[1]{\mathcal L(#1)}
\newcommand{\stairs}[1]{\mathcal S(#1)}
\newcommand{\hide}[1]{}

\DeclareMathOperator{\conv}{conv}

\DeclareMathOperator{\lcm}{lcm}
\DeclareMathOperator{\gens}{gens}
\DeclareMathOperator{\cone}{cone}
\DeclareMathOperator{\ass}{Ass}
\DeclareMathOperator{\maxass}{maxAss}

\DeclareMathOperator{\chr}{char}
\DeclareMathOperator{\spn}{span}

\begin{document}

\title[Symbolic Powers of Monomial Ideals]{Symbolic Powers of Monomial Ideals}

\author[S.M. Cooper]{Susan M. Cooper}
\address{Department of Mathematics\\
Central Michigan University\\
Mt. Pleasant, MI 48859 USA}
\email{s.cooper@cmich.edu}

\author[R.J.D. Embree]{Robert J.D. Embree}
\address{Department of Mathematics and Statistics\\
Queen's University\\
Kingston, ON K7L 3N6 Canada}
\email{3rjde@queenu.ca}

\author[H.T. H\`a]{Huy T\`ai H\`a}
\address{Department of Mathematics\\
Tulane University \\
New Orleans, LA 70118 USA}
\email{tai@math.tulane.edu}

\author[A.H. Hoefel]{Andrew H. Hoefel}
\address{Department of Mathematics and Statistics\\
Queen's University\\
Kingston, ON K7L 3N6 Canada}
\email{ahhoefel@mast.queensu.ca}


\date{\today}




\begin{abstract}
We investigate symbolic and regular powers of monomial ideals.  For a square-free monomial ideal $I \subseteq \field [x_0, \ldots, x_n]$ we show $I^{(t(m+e-1)-e+r)} \subseteq \maxideal^{(t-1)(e-1)+r-1}(I^{(m)})^t$ for all positive integers $m, t$ and $r$, where $e$ is the big-height of $I$ and $\maxideal = (x_0, \ldots, x_n)$.  This captures two conjectures ($r=1$ and $r=e$): one of Harbourne-Huneke and one of Bocci-Cooper-Harbourne.  We also introduce the symbolic polyhedron of a monomial ideal and use this to explore symbolic powers of non-square-free monomial ideals.  
\end{abstract}

\maketitle

\section{Introduction}

Comparing the behaviour of symbolic and regular powers of a homogeneous ideal has become key to understanding many problems in commutative algebra and algebraic geometry.  Investigations have involved many mathematicians and have touched a number of related areas such as number theory and complex functions.  Among the known results is a celebrated containment of Ein-Lazarsfeld-Smith \cite{refELS} and Hochster-Huneke \cite{refHoHu}.  Much effort has been put towards tightening this containment.  To be more precise, let $R = \field [\mathbb P^n] = \field [x_0, \ldots, x_n]$ where $\field$ is a field and $0 \not = I \subset R$ be a homogeneous ideal.  We define the \emph{$m$-th symbolic power of $I$} to be 
$$I^{(m)} = R \cap \bigcap_{\p \in {\ass}(I)} I^mR_{\p}.$$
We know that $I^r \subseteq I^{(m)}$ if and only if $r \geq m$ (see Lemma 8.1.4 of \cite{refBetal}).  However, determining which symbolic powers are contained in regular powers is a much more complicated problem.  Using multiplier ideals and tight closure, respectively, Ein-Lazarsfeld-Smith and Hochster-Huneke show that $I^{(rn)} \subseteq I^r$ for all $r > 0$.  Further, for some ideals (including monomial ideals) Harbourne showed that $I^{(rn-(n-1))} \subseteq I^r$ for all $r > 0$.  

With the goal of tightening up the above containments, Harbourne and Huneke list a number of conjectured containments in their paper \cite{refHaHu}.  In general, they focus on the following questions:

\begin{q}\cite[Question 1.3, 1.4 and Conjecture 4.1.5]{refHaHu}
Let $R = \field [\mathbb P^n]$ and $\maxideal = (x_0, \ldots, x_n)$ be the maximal homogeneous ideal of $R$.  Let $I \subseteq R$ be a homogeneous ideal.
\begin{enumerate}
\item For which $m, i$ and $j$ do we have $I^{(m)} \subseteq \maxideal^jI^i$?
\item For which $j$ does $I^{(rn)} \subseteq \maxideal^jI^r$ hold for all $r$?
\item If $I$ satisfies $I^{(rn-(n-1))} \subseteq I^r$, then for which $j$ does $I^{(rn-(n-1))} \subseteq \maxideal^jI^r$? 
\end{enumerate}
\end{q}

Such containments would imply bounds for the invariant $\alpha(I) = \min\{d \mid I_d \not = 0\}$.  For example, using the fact that $\alpha(I^r) = r \alpha(I)$, the containment $I^{(rn)} \subseteq I^r$ implies the bound of Waldschmidt-Skoda (see \cite{refW} and \cite{refSk}) that $\frac{\alpha(I^{(r)})}{r} \geq \frac{\alpha(I)}{n}$ when $I$ is the ideal of a finite set of points in $\mathbb P^n$.  Moreover, again for an ideal $I$ of points in $\mathbb P^n$, Chudnovsky \cite{refCh} conjectured the inequality $\frac{\alpha(I^{(r)})}{r} \geq \frac{\alpha(I) + n - 1}{n}$ which would follow if the conjectured containment $I^{(rn)} \subseteq \maxideal^{r(n-1)}I^r$ of Harbourne-Huneke held.

For convenience of the reader, we list the specific conjectures of Harbourne-Huneke in Section~\ref{conj}. 
Much of the work on these conjectures has been on families of points in $\mathbb P^n$.
The goal of this paper is to study whether these containments between symbolic and regular powers hold for monomial ideals. 
We start our investigation in Section \ref{sectionprimary} where we give a new formula for the symbolic powers of a monomial ideal as the intersection of regular powers of certain primary components of the ideal. This formula is independent of the choice of primary decomposition.

Section \ref{sqgen} is dedicated to containments restricted to square-free monomial ideals. Our main result is Theorem \ref{thm.sqfree}:

\begin{thm}
Let $I \subset \field [x_0, x_1, \ldots, x_n]$ be any square-free monomial ideal and let $\maxideal = (x_0, \ldots, x_n)$.  Then for all positive integers $m$, $t$ and $r$ we have the containment
$$I^{(t(m+e-1)-e+r)} \subseteq \maxideal^{(t-1)(e-1)+r-1}(I^{(m)})^t$$
where $e$ is the big-height of $I$ (i.e., the maximum of the heights of the associated primes of $I$).
\end{thm}
This theorem captures two conjectures: one of the conjectures of Harbourne-Huneke (namely, $I^{(rn-(n-1))} \subseteq \maxideal^{(r-1)(n-1)} I^r$, for reduced points in $\mathbb P^n$) and an additional conjecture of Bocci-Cooper-Harbourne. Most of the other conjectures follow for square-free monomial ideals as a result of these two. We conclude Section \ref{sqgen} with an example of a non-square-free monomial ideal which does not satisfy the containment $I^{(rn-(n-1))} \subseteq \maxideal^{(r-1)(n-1)} I^r$.

The last two sections of the paper are dedicated to a new tool for studying symbolic powers of monomial ideals. Section \ref{sectionsymbpoly} introduces the \emph{symbolic polyhedron}, a convex polyhedron in $\mathbb R^{n+1}$, which, when scaled by a factor of $m$, contains the exponent vectors of all monomials in $I^{(m)}$ (see Theorem~\ref{thm:symbcontainment}). In this way, the symbolic polyhedron approximates the symbolic powers of a monomial ideal. In Section \ref{sectionalpha}, we show the utility of the symbolic polyhedron; we show that a simple invariant of the symbolic polyhedron gives the Waldschmidt constant $\gamma(I) = \lim_{m \to \infty} I^{(m)}/m$. 

We also use the symbolic polyhedron to show the following containment 
for monomial ideals (Theorem \ref{thm:alphaslope}):

\begin{thm}
Suppose $\symbpoly$ is the symbolic polyhedron of a monomial ideal $I$ with big-height $e$. 
For all integers $r \geq 0$ and $m \geq \max(er, \beta(I^r)/\alpha(\symbpoly))$,
\[ I^{(m)} \subseteq \maxideal^{\lceil \alpha(\symbpoly)m \rceil - \beta(I^r) } I^r,\]
where $\alpha(\symbpoly) = \min\{ a_1 + \cdots + a_n \mid (a_1,\ldots,a_n) \in \symbpoly\}$ and $\beta(I^r)$ is the maximum of the degrees of the minimal generators of $I^r$.
\end{thm}

In certain situations we show that Chudnovsky's conjecture holds. Although Chudnovsky's conjecture is often thought of as a consequence of the containment $I^{(rn)} \subseteq \maxideal^{r(n-1)} I^r$, we conclude the paper by showing examples of the converse. That is, assuming that Chudnovsky's conjecture holds, we can apply Theorem \ref{thm:alphaslope} to show containments between regular and symbolic powers of monomial ideals (see Propositions \ref{prop:equigen} and \ref{prop:ifchudholds}).
\medskip

{\bf Acknowledgements.} We would like to thank Brian Harbourne for sharing his insights and giving his support on this project. Authors Robert Embree and Andrew Hoefel would like to thank Gregory G. Smith for introducing them to this project, for many helpful conversations, and for financial support. We also thank S. Amin Seyed Fakhari for pointing out an error in a statement of Carath\'eodory's Theorem that appeared in a previous version of this paper.

\section{The Conjectures}\label{conj}

In \cite{refHaHu}, Harbourne and Huneke present a series of conjectures relating containments of symbolic and regular powers of ideals of points in $\mathbb P^n$.  For the convenience of the reader, we list these conjectures below along with variants of Chudnovsky's conjecture.  This paper will focus on parallel conjectures for monomial ideals.

For what follows, we let $\field [\mathbb P^n] = \field [x_0, \ldots, x_n]$ and $\maxideal = (x_0, \ldots, x_n)$ be the maximal proper homogeneous ideal of  $\field[\mathbb P^n]$.  Observe that if $I = \{P_1, \ldots, P_s\} \subseteq \mathbb P^n$ is a finite set of distinct points, then $I^{(m)} = I(P_1)^m \cap \cdots \cap I(P_s)^m$.  Given non-negative integers $m_1, \ldots, m_s$, we define the {\it fat point scheme with multiplicities $m_1, \ldots, m_s$} to be the scheme in $\mathbb P^n$ defined by the ideal $J = I(P_1)^{m_1} \cap \cdots \cap I(P_s)^{m_s}$.  The $\ell$th symbolic power of the fat points ideal $J$ is then 
\[ J^{(\ell)} = I(P_1)^{\ell m_1} \cap \cdots \cap I(P_s)^{\ell m_s}.
\]

\begin{conj}[{\cite[Conjecture 2.1]{refHaHu}}]\label{conj1}
Let $I=\cap_{i=1}^n I(P_i)^{m_i}\subset \field [\mathbb P^n]$ be any fat points ideal.  Then $I^{(rn)}\subseteq \maxideal^{r(n-1)}I^r$
holds for all $r>0$. 
\end{conj}

\begin{conj}[{\cite[Conjecture 8.20]{refBetal}}]\label{Essenconj}
Let $I\subseteq \field [\mathbb P^n]$ be a homogeneous ideal.
Then $I^{(rn-(n-1))}\subseteq I^r$ holds for all $r$. 
\end{conj}

\begin{conj}[{\cite[Conjecture 4.1.4]{refHaHu}}]\label{p2conj}
Let $I\subseteq \field [\mathbb P^2]$ be the radical ideal of a finite set of points in $\mathbb P^2$.
Then $I^{(m)}\subseteq I^r$ holds whenever $m/r\ge 2\alpha(I)/(\alpha(I)+1)$. 
\end{conj}

\begin{conj}[{\cite[Conjecture 4.1.5]{refHaHu}}]\label{conj2}
Let $I\subseteq \field [\mathbb P^n]$ be the radical ideal of a finite set of points in $\mathbb P^n$.  Then $I^{(rn-(n-1))}\subseteq \maxideal^{(r-1)(n-1)}I^r$ holds for all $r \geq 1$. 
\end{conj}

\begin{conj}[{\cite[Conjecture 4.1.8]{refHaHu}}]\label{EvoEssenconj2}
Let $I\subseteq \field [\mathbb P^n]$ be the radical ideal of a finite set of points in $\mathbb P^n$.
Then for every $r>0$,
\[ \alpha(I^{(rn-(n-1))})\ge r\alpha(I)+(r-1)(n-1).
\]
\end{conj}

\begin{conj}[Chudnovsky's Conjecture {\cite{refCh}}]\label{conj:chud}
Let $I \subseteq \field [\mathbb P^n]$ be the radical ideal of a finite set of points. Then, for all $r>0$,
\[ 
\frac{\alpha(I) + n -1}{n} \leq
\frac{\alpha(I^{(r)}) }{r}.
\]
\end{conj}

\begin{conj}[{\cite[Question 4.2.1]{refHaHu}}]\label{refinedChud}
Let $I\subseteq \field [\mathbb P^n]$ be the radical ideal of a finite set of points in $\mathbb P^n$.
Then, for all $r > 0$,
\[ \frac{\alpha(I^{(m)})+n-1}{m+n-1}\leq \frac{\alpha(I^{(r)})}{r}.
\]
\end{conj}

\begin{conj}[{\cite[Question 4.2.2]{refHaHu}}]\label{conj3}
Let $I\subseteq \field [\mathbb P^n]$ be the radical ideal of a finite set of points in $\mathbb P^n$ for $n \geq 2$.  Then  
$I^{(t(m+n-1))}\subseteq \maxideal^t(I^{(m)})^t$.
\end{conj}

\begin{conj}[{\cite[Question 4.2.3]{refHaHu}}]\label{refinedEV2}
Let $I\subseteq \field [\mathbb P^n]$ be the radical ideal of a finite set of points in $\mathbb P^n$. Then 
$I^{(t(m+n-1))}\subseteq \maxideal^{t(n-1)}(I^{(m)})^t$.
\end{conj}

While working on proving these conjectures in special cases, Bocci-Cooper-Harbourne also introduce the following conjecture:

\begin{conj}[{\cite[Conjecture 3.9]{refBCH}}]\label{ourconj}\label{conj4}
Let $I\subseteq \field [\mathbb P^n]$ be the radical ideal of a finite set of points in $\mathbb P^n$. Then $I^{(t(m+n-1)-n+1)}\subseteq \maxideal^{(t-1)(n-1)}(I^{(m)})^t$
hold for all $m\geq 1$.
\end{conj}

\begin{rem}
The first counter-example to Conjecture~\ref{Essenconj} was due to Dumnicki-Szemberg-Tutaj-Gasi\'nska \cite{refDST}. Their counter-example broke the containment in $\mathbb C[\mathbb P^2]$ with $r = 2$ and inspired the counter-example of Bocci-Cooper-Harbourne \cite{refBCH} which works in $\field[\mathbb P^2]$ with $\chr \field = 3$ and $r= 2$. More recently, Harbourne and Seceleanu \cite{refHS} have produced counter-examples in $\field[\mathbb P^2]$ for arbitrary $r$ when $\chr \field \neq 2$. In the same paper, Harbourne and Seceleanu also produce counter-examples to Conjecture~\ref{Essenconj} in $\field[\mathbb P^n]$ for various $n$ and for values of $r$ determined by $n$ and $\chr \field$.
\end{rem}

\section{Symbolic Powers of Monomial Ideals} \label{sectionprimary}

In this section, we show that the symbolic powers of a monomial ideal can be described without using localizations.  Before we get started, we set some notation and recall some facts.  For the remainder of the paper, we fix $R = \field[x_0, \ldots, x_n]$ to be the polynomial ring in $n+1$ variables over a field $\field$.

Recall that every monomial ideal $I \subset R$ has a unique set of \emph{minimal monomial generators} $\gens(I)$ which are minimal in the sense that no proper subset of $\gens(I)$ generates $I$ and also in the sense that each monomial in $\gens(I)$ is minimal in the poset of monomials in $I$ ordered by divisibility.  Also, a monomial ideal $\p$ is a prime ideal if and only if $\gens(\p) \subseteq \{x_0, \ldots, x_n\}$ and a monomial ideal $Q$ is primary if and only if ${\gens(Q) =\{ x_{i_1}^{a_1} ,\ldots, x_{i_k}^{a_k}, m_1, \ldots, m_\ell\}}$ and each $m_j$ is a monomial in $x_{i_1}, \ldots, x_{i_k}$. Although a monomial ideal $I$ has an irredundant primary decomposition, it will be an important fact that $I$ has a (possibly redundant) primary decomposition $I = Q_1 \cap \cdots \cap Q_k$ where each $Q_i$ is a primary monomial ideal generated solely by powers of variables. We denote the set of associated primes of $I$ by $\ass(I)$. 

Our first step in showing that the symbolic powers of a monomial ideal can be described without using localizations is to drop unnecessary intersections in the definition of $I^{(m)}$. Let $\maxass(I)$ be the subset of $\ass(I)$ which consists of associated primes of $I$ that are maximal with respect to inclusion. 

\begin{lemma} \label{lem:symb1}
The $m$-th symbolic power of an ideal $I \subseteq R$ is 
\[ I^{(m)} = \bigcap_{\p \in \maxass(I)} (R \cap I^m R_\p).
\]
\begin{proof}
If $\p$ and $\p'$ are prime ideals with $\p' \subseteq \p$ then $I^{m} R_\p \subseteq I^{m} R_{\p'}$. Thus we can drop each term $R \cap I^{m} R_{\p'}$ with $\p' \in \ass(I) \setminus \maxass (I)$ from the intersection that defines $I^{(m)}$.
\end{proof}
\end{lemma}

\begin{lemma} \label{lem:symb2}
Suppose $I = Q_1 \cap \cdots \cap Q_k$ is a primary decomposition of an ideal $I \subseteq R$ and suppose $\p$ is an associated prime of $I$.
Let $Q_{\subseteq \p}$ be the intersection of all $Q_i$ with $\sqrt{Q_i} \subseteq \p$. For every $m$, $Q_{\subseteq \p}^m R_\p = I^m R_\p$.
\begin{proof} 
Since $I \subseteq Q_i$ for each $i=1, \ldots, k$, we have $I \subseteq Q_{\subseteq \p}$ and hence $I^m R_\p \subseteq Q_{\subseteq \p}^m R_\p$ for each $\p \in \ass(I)$.

For the reverse containment, take $f \in Q_{\subseteq \p}^m R_\p$ and express it as $f = \sum_{j=1}^\ell c_j f_j$ where $c_j \in R_\p$ and $f_j \in Q_{\subseteq \p}^m$ for all $j$. Let $\Lambda \subseteq \{1, \ldots, k \}$ be the set of all indices $i$ with $\sqrt{Q_i} \not\subseteq \p$. For each $i \in \Lambda$, pick $g_i \in \sqrt{Q_i} \setminus \p$ and let $h_i$ be a power of $g_i$ with $h_i \in Q_i$. Since $\p$ is prime, $h_i \notin \p$ and therefore $h_i$ is invertible in $R_\p$.   The product $f_j' = f_j \prod_{i \in \Lambda} h_i^m \in I^m$ since $f_j \in Q_{\subseteq \p}^m \subseteq Q_i^m$ for $i \notin \Lambda$ and $h_i^m \in Q_i^m$ for $i \in \Lambda$. Since $h_i$ is invertible in $R_\p$, $f= \sum_{j=1}^\ell \frac{c_j}{(\prod_{i \in \Lambda} h_i)^m} f_j'$ is in $I^m R_\p$.
\end{proof}
\end{lemma}

\begin{lemma} \label{lem:symb3}
Suppose $I$ is a monomial ideal with a unique maximal associated prime. For every $m$, $I^{(m)} = I^m$.

\begin{proof}
From the definition of the symbolic power, it is clear that $I^m \subseteq I^{(m)}$ for all $m$. 
We need to show the opposite inclusion $I^{(m)} \subseteq I^m$.
By Lemma \ref{lem:symb1}, $I^{(m)} = R \cap I^m R_\p$ where $\p$ is the unique maximal associated prime of $I$.

From \cite{MR0242802} Proposition 3.11, $R \cap I^m R_\p = \bigcup_{g \notin \p} (I^m: g)$. Thus, for any element $f \in R \cap I^m R_\p$ there is some $g \notin \p$ with $fg \in I^m$. We will use induction on the number of terms in $f$ to show that if $f \in (I^m:g)$ for some $g \notin \p$ then $f \in I^m$. If $f$ has no terms then $f = 0 \in I^m$.

Take a polynomial $f \in (I^m:g)$ for some $g \notin \p$. As $g \notin \p$ and $\p$ is a monomial ideal, some term of $g$ is not in $\p$. If we pick an elimination order for $\{ x_i \mid x_i \notin \p \}$, then the leading term $g'$ of $g$ is a monomial not in $\p$. Let $f'$ be the leading term of $f$ with respect to the same term order. The leading term of $fg$ is $f'g'$ and, in particular, it has a non-zero coefficient. Since $fg \in I^m$ and $I^m$ is a monomial ideal, its leading term $f'g'$ must be in $I^m$. 

Since $g' \notin \p$ and all generators of $I^m$ are monomials in the variables generating $\p$, we see that $f'g' \in I^m$ implies $f' \in I^m$.  Therefore, $f -f' \in (I^m:g)$. Since $f-f'$ has fewer terms than $f$, we have $f-f' \in I^m$ by our inductive hypothesis. Therefore, $f = (f-f') + f' \in I^m$ and hence $I^{(m)} \subseteq I^m$.
\end{proof}
\end{lemma}

\begin{defn}
The \emph{big-height} of an ideal $I \subseteq R$ is the maximum of the heights of its associated primes.
\end{defn}

Many of Harbourne and Huneke's conjectures on symbolic powers of ideals of points in $\mathbb P^n$ depend on $n$. In order to restate conjectures for monomial ideals, the role of $n$ will be replaced by the big-height the ideal.

\begin{rem} \label{rem:nplusone} If $I \subseteq R = \field[x_0,\ldots, x_n]$ is a monomial ideal with big-height $e =n+1$, then $\maxideal$ is an associated prime of $I$.
Thus $\maxideal$ is the unique maximal associated prime of $I$ and we know from Lemma \ref{lem:symb3} that the symbolic and regular powers of $I$ are equal. Consequently, we are only concerned with ideals having $e \leq n$.
\end{rem}

\begin{prop} \label{prop:qsubp}
Suppose $I = Q_1 \cap \cdots \cap Q_k$ is a primary decomposition of a monomial ideal and $\p$ is an associated prime of $I$.
Let $Q_{\subseteq \p}$ be the intersection of all $Q_i$ with $\sqrt{Q_i} \subseteq \p$. For all $m$,
$Q_{\subseteq \p}^m = R \cap I^m R_\p$.
\begin{proof}
By Lemma \ref{lem:symb2}, it suffices to show that $Q_{\subseteq \p}^m = R \cap Q_{\subseteq \p}^m R_\p$. Since $Q_{\subseteq \p}$ has a single maximal associated prime, Lemma \ref{lem:symb1} gives $Q_{\subseteq \p}^{(m)} = R \cap Q_{\subseteq \p}^m R_\p$. By Lemma \ref{lem:symb3}, we have $Q_{\subseteq \p}^{(m)} = Q_{\subseteq \p}^m$, completing the proof.
\end{proof}
\end{prop}

A consequence of Proposition \ref{prop:qsubp} with $m=1$ is that $Q_{\subseteq \p} = R \cap I R_\p$ and thus $Q_{\subseteq \p}$ does not depend on a choice of primary decomposition. Henceforth, we will use $Q_{\subseteq \p} = R \cap I R_\p$ as the definition of $Q_{\subseteq \p}$ so that we may avoid choosing   a primary decomposition.

\begin{thm} \label{thm:symbpower}
The $m$-th symbolic power of a monomial ideal $I$ is
\[ I^{(m)} = \bigcap_{\p \in \maxass(I)} Q_{\subseteq \p}^{m}
\]
where $Q_{\subseteq \p} = R \cap I R_\p$.
\begin{proof}
This follows from Lemma \ref{lem:symb1} and Proposition \ref{prop:qsubp}.
\end{proof}
\end{thm}

The above theorem generalizes the following well-known result for ideals without embedded primes: if $I = Q_1 \cap \cdots \cap Q_k$ is an intersection of primary monomial ideals with incomparable radicals, then $I^{(m)} = Q_1^m \cap \cdots \cap Q_k^m$ (see \cite[Exercise 5.1.29]{MR1800904}). In particular, Theorem~\ref{thm:symbpower} is already known for square-free monomial ideals. In Sections \ref{sectionsymbpoly} and \ref{sectionalpha}, we will prove containments for arbitrary monomial ideals using this theorem.

We conclude this section with two containments that compare different symbolic powers of monomial ideals.  

\begin{prop}  \label{prop:onemaxideal}
If $I$ is a monomial ideal then $I^{(r+1)} \subseteq \maxideal I^{(r)}$. 
\begin{proof}[Proof]
Let $f$ be a monomial in $I^{(r+1)}$ and let $x_i$ be a variable dividing $f$. Using Theorem \ref{thm:symbpower}, we know $I^{(r+1)} = \bigcap_{\p \in \maxass(I)} Q_{\subseteq \p}^{r+1}$, for each $\p \in \maxass(I)$. So for each particular $\p \in \maxass(I)$, we can express $f$ as $f = f_1 \cdots f_{r+1} g$ for monomials $f_i \in Q_{\subseteq \p}$ and a monomial $g \in R$. Either $x_i$ divides some $f_{j}$, giving $f/x_i = f_1 \cdots \widehat{f_j} \cdots f_{r+1} (g f_{j}/x_i) \in Q_{\subseteq \p}^{r}$ or $x_i$ divides $g$ and $f/x_i \in Q_{\subseteq \p}^{r+1} \subset Q_{\subseteq \p}^{r}$.  Since this argument works for all $\p \in \maxass(I)$, we have $f \in \maxideal I^{(r)}$ and hence $I^{(r+1)} \subseteq \maxideal I^{(r)}$.
\end{proof}
\end{prop}

Proposition \ref{prop:onemaxideal} shows that $I^{(r+s)} \subseteq \maxideal^{s} I^{(r)}$ for any monomial ideal $I$ and positive integers $r$ and $s$. 
In particular, we may prove containments of the form of Conjectures \ref{conj1} and \ref{conj2} with the regular power of $I$ replaced with a symbolic power (e.g., $I^{(er)} \subseteq \maxideal^{r(e-1)} I^{(r)}$).

Often Proposition \ref{prop:onemaxideal} is not optimal in the sense that $I^{(r+1)}$ may be contained in $\maxideal^s I^{(r)}$ for some $s > 1$. 

\begin{prop} \label{prop:sigmamaxideals}
Let $I$ be a monomial ideal with big-height $e$ and let $\sigma(I)$ be the minimum number of variables in the support of a monomial in $I$. 
For any integer $r \geq 1$, $I^{(r+e)} \subseteq \maxideal^{\sigma(I)} I^{(r)}$.
\begin{proof}
Take a monomial $f \in I^{(r+e)}$. Let $h$ be the product of the variables dividing $f$.
In other words, $h$ is the square-free part of $f$.

Since $I^{(r+e)} \subseteq Q_{\subseteq \p}^{r+e}$ for each $\p \in \maxass(I)$, we can express $f$ as
$f = f_1 \cdots f_{r+e}g$ for monomials $f_i \in \gens(Q_{\subseteq \p})$ and a monomial $g \in R$. Since each $f_i$ is a generator of $Q_{\subseteq \p}$, it is a product of variables in $\p$. Thus any product of variables not in $\p$ which divides $f$ must divide $g$.  In particular, the product of the variables in $h$ which are not in $\p$ must divide $g$.

Since $h$ is square-free and $\height \p \leq e$, there are at most $e$ variables in the support of $h$ that are in $\p$.
These variables are factors of at most $e$ of the $f_i$. Therefore $f/h \in Q_{\subseteq\p}^r$ and hence $f \in \maxideal^{\deg h} I^{(r)}$. As $f \in I^{(r+e)} \subseteq I$, we know $\deg h \geq \sigma(I)$ and hence $f \in \maxideal^{\sigma(I)} I^{(r)}$.  As $f$ was an arbitrary monomial in $I^{(r+e)}$ we have the desired containment.
\end{proof}
\end{prop}
If $\sigma(I) > e$ then Proposition \ref{prop:sigmamaxideals} is an improvement on Proposition \ref{prop:onemaxideal}. This occurs when $I$ has many disjoint associated primes.

\begin{ex} The ideal 
\begin{align*}
	I 
	&= (xyz, xyw, xzw, yzw) \\
	&= (x,y) \cap (x,z) \cap (x,w) \cap (y,z) \cap (y,w) \cap (z,w)
\end{align*}
has big-height $e=2$ and $\sigma(I) = 3$. Thus, $I^{(r+2)} \subseteq \maxideal^3 I^{(r)}$ for all $r$.
\end{ex}


\section{A General Statement for Square-Free Monomial Ideals}\label{sqgen}

In this section we capture Conjectures~\ref{refinedEV2} and \ref{conj4} for square-free monomial ideals by proving a more general statement.  

\begin{rem}
Proposition 3.10 of \cite{refBCH} proves the following implications:
\begin{itemize}
\item[(1)] Conjecture \ref{conj2} implies Conjecture \ref{EvoEssenconj2}.
\item[(2)] Conjecture \ref{refinedEV2} implies Conjectures \ref{conj:chud}, \ref{refinedChud} and \ref{conj3}, and, when $I$ is radical, Conjecture \ref{conj1}.
\item[(3)] Conjecture \ref{ourconj} implies Conjectures \ref{conj2} and \ref{EvoEssenconj2},
and, when $I$ is the radical ideal of a finite set of points, Conjecture \ref{Essenconj}.
\end{itemize}
These implications are also true for monomial ideals.  The proofs are the same as given in \cite{refBCH}.  So, once we prove Conjectures~\ref{refinedEV2} and \ref{conj4}, we will have verified all of the conjectures of Harbourne-Huneke with the exception of Conjecture~\ref{p2conj} for square-free monomial ideals. We will also have proven Chudnovsky's conjecture for square-free monomial ideals (Conjecture \ref{conj:chud}).
\end{rem}

Recall that all of our ideals will be homogeneous ideals in $\field[x_0,\ldots,x_n]$ and that $\maxideal = (x_0,\ldots,x_n)$.
\begin{thm}\label{thm.sqfree}
Suppose $I$ is a square-free monomial ideal.  Then for all positive integers $m$, $t$ and $r$ we have the containment
$$I^{(t(m+e-1)-e+r)} \subseteq \maxideal^{(t-1)(e-1)+r-1}(I^{(m)})^t$$
where $e$ is the big-height of $I$.
\end{thm}

\begin{proof}
Let $I = Q_1 \cap Q_2 \cap \cdots \cap Q_s$ be the primary decomposition of $I$.  Since $I$ is square-free, each ideal $Q_i$ is prime and has the form $Q_i = (x_{i_1}, \ldots, x_{i_{n_i}})$.  Let $E_i$ denote set of indices of $Q_i$, namely $E_i = \{i_1, \ldots, i_{n_i}\}$. Then
$$e = \max\{|E_i| : \, \, \mbox{$Q_i$ a primary component of $I$}\} \geq |E_i| \ \forall i.$$

Observe that
$$I^{(t(m+e-1)-e+r)} = Q_1^{t(m+e-1)-e+r} \cap  Q_2^{t(m+e-1)-e+r} \cap \cdots  \cap Q_s^{t(m+e-1)-e+r}.$$
Take any monomial $x_0^{a_0}x_1^{a_1} \cdots x_n^{a_n} \in I^{(t(m+e-1)-e+r)}$.  Then for any primary component $Q_i$ of $I$, we have
$$\sum_{j \in E_i}a_j \geq t(m+e-1)-e+r.$$
For each $a_j$, choose $d_j \in \ZZ_{\ge 0}$ such that $d_j t \leq a_j < (d_j+1)t$.  Let $a_j' = a_j - d_jt \leq t-1$. It can be seen that
$\sum_{j \in E_i} a_j' \leq (t-1)|E_i| \leq (t-1)e.$
Thus,
$$\sum_{j \in E_i}d_jt \ge t(m+e-1)-e+r - (t-1)e = t(m-1)+r.$$
Since the left hand side of the inequality is divisible by $t$, this implies that for any primary component $Q_i$ of $I$, we have
$$\sum_{j \in E_i} d_jt \geq tm, \text{ i.e., } \sum_{j \in E_i} d_j \ge m.$$

Now consider the system of inequalities $\{\sum_{j \in E_i} d_j \ge m ~|~ i=1, \dots, s\}$. By successively reducing the values of $d_j$'s (keeping them non-negative), we can choose $0 \leq d_j' \leq d_j$ such that the system of inequalities
\begin{align}\label{eq.system3}
\big\{\sum_{j \in E_i} d_j' \geq m ~|~ i = 1, \dots, s\big\} 
\end{align}
still holds, but for at least one value of $i$ we obtain an equality. Assume that $\sum_{j \in E_{\ell}} d_j' = m$ for some $1 \le \ell \le s$.

Let
\[ f =\left(\prod_{j=0}^n x_j^{d_j'}\right)^t \text{ and }
g =\prod_{j=0}^n x_j^{a_j-d_j't}.
\]
It can be seen that
$x_0^{a_0} \cdots x_n^{a_n} = fg$. Also, by (\ref{eq.system3}),
\[ \prod_{j=0}^n x_j^{d_j'} \in Q_1^m \cap Q_2^m \cap \cdots \cap Q_s^m,
\]
which in turn implies that $\prod_{j=1}^n x_j^{d_j'} \in I^{(m)}$. Thus, $f \in (I^{(m)})^t$.

We claim that $g \in \maxideal^{(t-1)(e-1)+r-1}$. Indeed, notice that $g$ is divisible by
\[\prod_{j \in E_{\ell}}x_j^{a_j}\big/\big(\prod_{j \in E_{\ell}}x_j^{d_j'}\big)^t.
\]
Furthermore,
$$\deg\left(\prod_{j \in E_{\ell}}x_j^{a_j}\right) = \sum_{j \in E_{\ell}}a_j \geq t(m+e-1)-e+r,$$
and
$$\deg\left(\left(\prod_{j \in E_{\ell}}x_j^{d_j'}\right)^t\right) = t \sum_{j \in E_{\ell}}d_j' = tm.$$
Hence,
$$\deg(g) \geq t(m+e-1)-e+r - tm = (t-1)(e-1)+r-1,$$
and so $g \in \maxideal^{(t-1)(e-1)+r-1}$. Therefore,
$$x_0^{a_0} \cdots x_n^{a_n} \in \maxideal^{(t-1)(e-1)+r-1}(I^{(m)})^t.$$
Since this is true for all monomials $x_0^{a_0} \cdots x_n^{a_n}$ in $I^{(t(m+e-1)-e+r)}$ and $I^{(t(m+e-1)-e+r)}$ is a monomial ideal, we conclude that $I^{(t(m+e-1)-e+r)} \subseteq \maxideal^{(t-1)(e-1)+r-1}(I^{(m)})^t$ as desired.
\end{proof}

When $r=1$ and $r=e$, we capture the two most general conjectures of Harbourne-Huneke and Bocci-Cooper-Harbourne (namely, Conjecture~\ref{refinedEV2} for $r=e$ and Conjecture~\ref{conj4} for $r=1$) for square-free monomial ideals.

\begin{cor}[Conjecture \ref{refinedEV2}]\label{conj3m}
If $I$ is a square-free monomial ideal, then for all positive integers $m$ and $t$ we have the containment
$$I^{(t(m+e-1))} \subseteq \maxideal^{t(e-1)}(I^{(m)})^t$$
where $e$ is the big-height of $I$.
\end{cor}

\noindent
\begin{proof}
The containment follows immediately from Theorem~\ref{thm.sqfree} by letting $r=e$.
\end{proof}

\begin{cor}[Conjecture \ref{conj4}]\label{conj4m}
If $I$ is a square-free monomial ideal, then for all positive integers $m$ and $t$ we have the containment
$$I^{(t(m+e-1)-e+1)} \subseteq \maxideal^{(t-1)(e-1)}(I^{(m)})^t$$
where $e$ is the big-height of $I$.
\end{cor}

\noindent
\begin{proof}
This containment is obtained by letting $r=1$ in Theorem~\ref{thm.sqfree}.
\end{proof}

\begin{rem} \label{rem.monomial}
For monomial ideals in general, if for any $j$ the powers of $x_j$ in the primary components $Q_i$'s are the same, then the same containment as in Theorem \ref{thm.sqfree} holds. The proof is exactly the same after performing a change of variables.
\end{rem}

\begin{ex}
Conjecture~\ref{conj2}, and consequently Conjecture \ref{conj4}, does not hold for all non-square-free monomial ideals. For example, let $I = (xy^2, yz^2, zx^2, xyz) = (x^2,y) \cap (y^2, z) \cap (z^2,x) \subset \field [x, y, z]$. When $r=2$, Conjecture ~\ref{conj2} suggests that $I^{(3)} \subseteq \maxideal I^2$.
The monomial $x^2y^2z^2$ is in $(x^2,y)^3 \cap (y^2,z)^3 \cap (z^2, x)^3$ and hence $x^2y^2z^2 \in I^{(3)}$ by Theorem \ref{thm:symbpower}. Although $x^2y^2z^2 \in I^2$, since $xyz$ is a generator of $I$, we do not have $x^2y^2z^2 \in \maxideal I^2$. Thus $I^{(3)} \not\subseteq \maxideal I^2$.
\end{ex}

\begin{rem}
Dumnicki \cite{refDumA} proved that $I^{(rn-(n-1))} \subseteq \maxideal^{(r-1)(n-1)}I^r$ for the ideal $I$ of generic points in when $n=3$. In addition, he shows in \cite{refDumA} and \cite{refDumB} that $I^{(rn)} \subseteq \maxideal^{r(n-1)}I^r$ for the ideal $I$ of $s$ fundamental points in $\mathbb P^n$ (i.e., points that lie at the coordinates $[1{\,:\,}0{\,:\,}\cdots{\,:\,}0], [0{\,:\,}1{\,:\,}0{\,:\,}\cdots{\,:\,}0], \ldots$). In this case where $I$ is the ideal of $s$ fundamental points in $\mathbb P^n$, $I$ is a square-free monomial ideal with big-height $n$, and so this case is covered by Corollary \ref{conj3m} with $t = r$ and $m=1$.
\end{rem}


\section{The Symbolic Polyhedron} \label{sectionsymbpoly}

In this section we define a polyhedron $\symbpoly(I) \subseteq \mathbb R^{n+1}$ which, when scaled by a factor of $m$, contains all lattice points $(a_0,a_1,\ldots, a_n)$ with $x_0^{a_0} x_1^{a_1} \cdots x_n^{a_n} \in I^{(m)}$. 

Let $\mathbb N$ be the set of non-negative integers. For each $\abar = (a_0, \ldots, a_n) \in \mathbb N^{n+1}$, we let $\xbar^\abar = x_0^{a_0} x_1^{a_1} \cdots x_n^{a_n}$ denote the monomial with exponent vector $\abar$. Let $\mathbb R_+$ be the set of non-negative real numbers. In particular, $\mathbb R_+^{n+1}$ is the positive orthant in $\mathbb R^{n+1}$. If $\abar = (a_0, \ldots, a_n)$ and $\bbar = (b_0, \ldots, b_n)$ are two vectors in $\mathbb R^{n+1}$ then we write $\abar \leq \bbar$ if $a_i \leq b_i$ for all $i$. This is gives partial ordering on $\mathbb R^{n+1}$ and the subsets $\mathbb R_+^{n+1}$ and $\mathbb N^{n+1}$. The monomials in $R$ are partially ordered by divisibility and we take the convention that $\xbar^\abar \leq \xbar^\bbar$ when $\xbar^\abar$ divides $\xbar^\bbar$ or, equivalently, if $\abar \leq \bbar$. 

A subset $\poly \subseteq \mathbb R^{n+1}$ is \emph{convex} if $\poly$ contains the line segment between any two points in $\poly$.
A linear combination $\lambda_1 \abar_1 + \cdots + \lambda_k \abar_k$ of vectors $\abar_i \in \mathbb R^{n+1}$ is called a \emph{conical combination}
if each $\lambda_i \geq 0$ and a \emph{convex combination} if it is a conical combination with $\sum_{i=1}^k \lambda_i = 1$.
The \emph{convex hull} of $A \subseteq \mathbb R^{n+1}$ is the smallest convex set containing $A$ and is given by the set all convex combinations of vectors in $A$. The \emph{conical hull} of $A$ is the set of all conical combinations of vectors in $A$. We denote the convex hull and conical hull of $A$ by $\conv(A)$ and $\cone(A)$, respectively.

A \emph{polyhedron} $\poly$ is a convex subset of $\mathbb R^{n+1}$ which is the intersection of a finite number of halfspaces, i.e.,  there exists $\nbar_1 , \ldots, \nbar_{k} \in \mathbb R^{n+1}$ and $c_1, \ldots, c_k \in \mathbb R$ with 
\[  \poly = \{ \abar \in \mathbb R^{n+1}  \mid \nbar_i \cdot \abar \leq c_i \text{ for $i = 1, \ldots, k$}\}.
\]
If $c_1 = \cdots = c_k = 0$ then we call $\poly$ a \emph{cone}. A polyhedron $\poly$ is a \emph{polytope} if it is bounded.

The \emph{Minkowski sum} of  $A,B \subseteq \mathbb R^{n+1}$ is
$A + B = \{ \abar + \bbar \in \mathbb R^{n+1} \mid \abar \in A, \bbar \in B\}$. On occasion, we will abuse notation and write $\abar + B$ for the Minkowski sum $\{ \abar \} + B$ where $\abar \in \mathbb R^{n+1}$ and $B \subseteq \mathbb R^{n+1}$.

Motzkin showed that every polytope is the convex hull of a finite number of points and every cone is the conical hull of a finite number of points \cite{MR693096}.
Furthermore, any polyhedron is the Minkowski sum of a unique cone, called the \emph{recession cone}, and a (non-unique) polytope. The recession cone of a polyhedron $\poly$ is given by
\[ C(\poly) = \{ \abar \in \mathbb R^{n+1} \mid \bbar + \lambda \abar \in \poly \text{ for all } \bbar \in \poly, \lambda \geq 0\}.
\]
If $\poly = \conv(\abar_1, \ldots, \abar_k)$ is a polytope and the convex hull of any proper subset of $\abar_1, \ldots, \abar_k$ is a proper subset of $\poly$ then we call $\abar_1, \ldots, \abar_k$ the \emph{vertices} of $\poly$. Similarly, if $\poly = \cone(\rbar_1, \ldots, \rbar_k)$ is a cone and the conical hull of any subset of $\rbar_1, \ldots, \rbar_k$ is a proper subset of $\poly$ then we call $\rbar_1, \ldots, \rbar_k$ the \emph{rays} of $\poly$.

We will need a version of Carath\'eodory's Theorem for polyhedra with non-zero recession cones.

\begin{thm}[Carath\'eodory's Theorem]
Let $\poly = \mathcal Q + C(\poly) \subseteq \mathbb R^{m}$ be a polyhedron for which $\mathcal Q$ is a polytope and $C(\poly)$ is the recession cone of $\poly$.
If $C(\poly) \neq \{ 0 \}$ then every point in $\poly$ can be expressed as the sum of a convex combination of at most $m$ vertices of $\mathcal Q$ and a conical combination of the rays of $C(\poly)$.

\begin{proof}
Carath\'eodory's Theorem \cite[Proposition 1.15]{MR1311028}
tells us that a point in an $n$-dimensional polytope is a convex combination of at most $n+1$ vertices of the polytope. So, we are done if $\dim \mathcal Q < m$. 

Assume that $\mathcal Q$ is $m$-dimensional and express $\abar \in \poly$ as $\abar = \bbar + \rbar$ for $\bbar \in \mathcal Q$ and $\rbar \in C(\poly)$. Take $\sbar \in C(\poly)$ with $\sbar \neq 0$.
As $\mathcal Q$ is compact, there exists some $\lambda \in \mathbb R$ with $\bbar'= \bbar - \lambda \sbar$ on the boundary of $\mathcal Q$. Furthermore, $\abar$ can be expressed as $\abar = \bbar' + (\lambda \sbar + \rbar)$ where $\lambda \sbar + \rbar \in C(\poly)$.

Since $\bbar'$ is on the boundary of $\mathcal Q$ and $\mathcal Q$ is full dimensional, $\bbar'$ is not in the relative interior of $\mathcal Q$ and hence $\bbar'$ is contained in a proper face of $\mathcal Q$. Thus, $\bbar'$ can be written as a convex combination of at most $m$ vertices of the $(m-1)$-dimensional (or lower) face. Since the vertices of a face of $\mathcal Q$ are vertices of $\mathcal Q$, we are done.
\end{proof}
\end{thm}

We let $\lattice M  = \{ \abar \in \mathbb N^{n+1} \mid \xbar^\abar \in M\}$ be the set of lattice points that are the exponent vectors of monomials in the monomial ideal or finite set of monomials $M$. 

\begin{lemma} \label{lemma:cone}
Suppose $I$ is a non-zero monomial ideal. The convex hull of $\lattice I$ is
\[ \conv(\lattice I) = \conv( \lattice{\gens(I)}) + \mathbb R^{n+1}_+.
\]
In particular, the recession cone of $\conv(\lattice I)$ is $\mathbb R^{n+1}_+$.

\begin{proof}
Since each $\abar \in \lattice I$ is in $\conv(\lattice{\gens(I)}) + \mathbb R^{n+1}_+$ and $\conv(\lattice{I})$ is the smallest convex set containing $\lattice I$, we know $\conv(\lattice I) \subseteq \conv(\lattice{\gens(I)}) + \mathbb R^{n+1}_+$.

Let $\gens(I) = \{ \xbar^{\abar_1}, \ldots, \xbar^{\abar_k} \}$. If $\bbar \in \conv(\lattice{\gens(I)}) + \mathbb R^{n+1}_+$ then $\bbar = \lambda_1 \abar_1 + \cdots + \lambda_k \abar_k + \rbar$ where $\lambda_1 + \cdots + \lambda_k = 1$, $\lambda_i \geq 0$ and $\rbar \in \mathbb R^{n+1}_+$. As $\lambda_1 + \cdots + \lambda_k = 1$, we know there is some index $j$ with $\lambda_j \neq 0$. Thus, 
\[ \bbar = \lambda_1 \abar_1 + \cdots + \lambda_{j-1} \abar_{j-1} +  \lambda_j \left(\abar_j + \frac{1}{\lambda_j}\rbar\right) + \lambda_{j+1} \abar_{j+1} + \cdots + \lambda_k \abar_k
\] 
and hence $\bbar \in \conv\left((\abar_1 + \mathbb R^{n+1}_+) \cup \cdots \cup (\abar_k + \mathbb R^{n+1}_+) \right)$. That is,
\[ \conv(\lattice{\gens(I)}) + \mathbb R^{n+1}_+ 
\subseteq \conv\left((\abar_1 + \mathbb R^{n+1}_+) \cup \cdots \cup (\abar_k + \mathbb R^{n+1}_+) \right).
\]

One can show by induction on $n$ that the positive orthant $\mathbb R^{n+1}_+$ is the convex hull of $\mathbb N^{n+1}$. Consequently, the translated orthant $\abar_i  + \mathbb R^{n+1}_+ = \{ \abar_i + \rbar \mid \rbar \in \mathbb R^{n+1}_+\}$ can be expressed as $\abar_i + \mathbb R^{n+1}_+ = \conv(\abar_i +  \mathbb N^{n+1})$.

Since $\conv(\lattice I)$ contains $\abar_i + \mathbb N^{n+1}$ for each $i$, it contains $\abar_i + \mathbb R^{n+1}_+$ for each $i$. Therefore $\conv(\lattice I)$ contains $(\abar_1 + \mathbb R^{n+1}_+) \cup \cdots \cup (\abar_k + \mathbb R^{n+1}_+)$ and hence its convex hull.
Thus, $\conv(\lattice{\gens(I)}) + \mathbb R^{n+1}_+ 
\subseteq \conv(\lattice I)$.
\end{proof}
\end{lemma}

\begin{defn} The \emph{symbolic polyhedron} of a monomial ideal $I$ is the polyhedron $\symbpoly \subseteq \mathbb R^{n+1}$ given by
\[ \symbpoly =  \bigcap_{\p \in \maxass(I)} \conv(\lattice{Q_{\subseteq \p}})
\]
where $Q_{\subseteq \p} = R \cap I R_\p$.
\end{defn}

\begin{thm}  \label{thm:symbcontainment}
If $\symbpoly$ is the symbolic polyhedron of a monomial ideal $I$ then $\lattice{I^{(m)}} \subset m\symbpoly$ for all integers $m \geq 1$.
\begin{proof}
From Theorem \ref{thm:symbpower}, $I^{(m)} = \bigcap_{\p \in \maxass(I)} Q_{\subseteq \p}^m$. If $\abar \in \lattice{I^{(m)}}$ then $\abar \in \lattice{Q_{\subseteq \p}^m}$ for each $\p \in \maxass(I)$ and hence $\abar = \sum_{i=1}^m \bbar_i + \rbar$ where $\rbar \in \mathbb N^{n+1}$ and $\xbar^{\bbar_i} \in \gens(Q_{\subseteq \p})$. By Lemma~\ref{lemma:cone}, $\sum_{i=1}^m  \frac{1}{m} \bbar_i + \frac{1}{m} \rbar \in \conv(\lattice{Q_{\subseteq \p}})$ and therefore $\abar \in m \conv(\lattice{Q_{\subseteq \p}})$ for each $\p$. Since $m \symbpoly = \bigcap_{\p \in \maxass(I)} m \conv(Q_{\subseteq \p})$, we have $\abar \in m \symbpoly$.
\end{proof}
\end{thm}


\section{The Waldschmidt Constant of a Monomial Ideal} \label{sectionalpha}

The \emph{Waldschmidt constant} of a non-zero homogeneous ideal $I \subseteq R = \field[x_0,\ldots, x_n]$ is given by 
\[ \gamma(I) = \lim_{m \to \infty} \alpha(I^{(m)})/m.
\]
This limit always exists and $\alpha(I^{(m)})/m \geq \gamma(I)$ for all $m$ (see \cite[Lemma 2.3.1]{refBH}).

In this section, we show that if $\symbpoly$ is the symbolic polyhedron of an arbitrary monomial ideal $I$, then $\alpha(\symbpoly) = \gamma(I)$. Thus, this invariant of $I$ can be extracted from the symbolic polyhedron. 
Finally, we restate Chudnovsky's conjecture for monomial ideals and, assuming that Chudnovsky's conjecture holds, we prove that Conjecture \ref{conj1} would follow in certain instances. Moreover, we give some families of monomial ideals for which Chudnovsky's conjecture holds.

%

\begin{prop} \label{prop:approaches}
Let $\symbpoly$ be the symbolic polyhedron of a monomial ideal $I$.
For any $\abar \in \symbpoly \cap \mathbb{Q}^{n+1}$ there exists a positive integer $b$ such that $\xbar^{m\abar} \in I^{(m)}$ whenever $m$ is divisible by $b$.

\begin{proof}
Fix a maximal associated prime $P$ of $I$ and let $Q_{\subseteq P} = R \cap I R_p$. Applying Lemma \ref{lemma:cone} to the ideal $Q_{\subseteq P}$ we can express $\conv(\lattice{Q_{\subseteq P}})$ as 
\begin{align*} 
	\conv(\lattice{Q_{\subseteq P}}) 
	&= \conv(\lattice{\gens(Q_{\subseteq P})}) + \mathbb R_{+}^{n+1}\\
	&= \conv(\lattice{\gens(Q_{\subseteq P})}) + \cone\{ \ebar_i \mid x_i \in P\} + \cone\{ \ebar_i \mid x_i \notin P\}.
\end{align*}
Let $h$ be the height of $P$. 
Let $V \subseteq \mathbb R^{n+1}$ be the $h$-dimensional coordinate subspace $V = \spn \{\ebar_i \mid x_i \in P\}$, where $\ebar_i$ are standard basis vectors.

Since $\abar \in \symbpoly \subseteq \conv(\lattice{Q_{\subseteq P}})$,  we
have $\abar = \bbar + \rbar$ where $\bbar \in \conv(\lattice{\gens(Q_{\subseteq P})}) + \cone\{ \ebar_i \mid x_i \in P\}$ 
and $\rbar \in \cone\{ \ebar_i \mid x_i \notin P\}$.
We may apply our version of Carath\'eodory's Theorem to $\bbar \in \conv(\lattice{\gens(Q_{\subseteq P})})  + \cone\{\ebar_i \mid x_i \in P\} \subset V$ to obtain
\[
	\bbar = 
	\sum_{i=1}^h \lambda_{P, i} \vbar_{P,i} 
	+ \sum_{\substack{j=0\\x_j \in P}}^n c_{P,j} \ebar_j 
\]
where the first summation is a convex combination of vertices 
$\conv(\lattice{\gens(Q_{\subseteq P})})$ and the second summation is conical combination of the basis vectors $\ebar_j$ with $x_j \in P$.
Consequently, we may write $\abar$ as
\[ 
	\abar 
	= \bbar + \rbar 
	= \sum_{i=1}^h \lambda_{P, i} \vbar_{P,i} + \sum_{j=0}^n c_{P,j} \ebar_j 
\]
where each $c_{P, j} \geq 0$.

Pick $b_P \in \mathbb N$ such that $b_P \lambda_{P,i}$ and $b_P c_{P,j}$ are integers for all $i, j$.  Let $b = \lcm_{P \in \maxass(I)} b_P$ and let $m$ be a positive integer divisible by $b$. As $\vbar_{P,i}$ is a vertex of $\conv(\lattice{\gens(Q_{\subseteq P})})$, $\xbar^{\vbar_{P,i}} \in \gens(Q_{\subseteq P})$. Since each $m \lambda_{P,i}$ is a positive integer and $\sum_{i=1}^h m\lambda_{P,i} = m$, we have $m \abar \geq \sum_{i=1}^h m \lambda_{P, i} \vbar_{P, i}$ and hence $\xbar^{m\abar} \in Q_{\subseteq P}^m$. Therefore $\xbar^{m\abar} \in \bigcap_{P \in \maxass(I)} Q_{\subseteq P}^m = I^{(m)}$ by Theorem \ref{thm:symbpower}.
\end{proof}
\end{prop}

For any polyhedron $\mathcal P \subseteq \mathbb R_{+}^{n+1}$, we define $\alpha(P) = \min \{ \abar \cdot \onebar \mid \abar \in \mathcal P \}$ where $\onebar = (1,\ldots, 1)$. Similarly, if $I$ is a homogeneous ideal, recall that $\alpha(I)$ is the least degree $d$ with $I_d \neq 0$.

\begin{cor} \label{cor:alphaq} 
If $\symbpoly$ is the symbolic polyhedron of a monomial ideal $I$, then for all $m \in \mathbb N$, $\alpha(I^{(m)}) \geq m\alpha(\symbpoly)$ and equality holds for infinitely many $m$.
\begin{proof}
If $\xbar^\bbar$ is a monomial in $I^{(m)}$ then $\bbar \in m\symbpoly$ by Theorem \ref{thm:symbcontainment}. Thus, $\alpha(I^{(m)}) \geq \alpha(m\symbpoly) = m \alpha(\symbpoly)$ for all $m$. If $\abar \in \symbpoly$ is a point with $\abar \cdot \onebar = \alpha(\symbpoly)$ then, by Lemma \ref{prop:approaches}, $\alpha(I^{(m)}) \leq m \abar \cdot \onebar = m \alpha(\symbpoly)$ for infinitely many $m$.
\end{proof}
\end{cor}

\begin{cor} \label{cor:wald}
If we suppose $\symbpoly$ is the symbolic polyhedron of a monomial ideal $I$ then $\alpha(\symbpoly) = \gamma(I)$.
\begin{proof}
As $\alpha(\symbpoly) = \alpha(I^{(m)})/m$ for infinitely many $m$, we know that $\alpha(\symbpoly)$ must equal the limit $\gamma(I) = \lim_{m \to \infty} \alpha(I^{(m)})/m$.
\end{proof}
\end{cor}

Let $\stairs I = \lattice I + \mathbb R^{n+1}_{+} = \{ \abar \in \mathbb R_{+}^{n+1} \mid \exists\,\xbar^\bbar \in \gens(I),\, \abar \geq \bbar \}$ be the staircase of the monomial $I \subseteq R$. For example, 
\[ \stairs{x^2, y^3} = \{ (a,b) \mid a\geq 2,\, b \geq 0 \} \cup \{ (a,b) \mid a \geq 0,\, b \geq 3\}.
\]

\begin{prop} \label{prop:erQ} 
If $\symbpoly$ is the symbolic polyhedron of a monomial ideal $I$ with big-height $e$, then $er\symbpoly \subseteq \stairs{I^r}$ for all $r$.
\begin{proof}
Let $P \in \maxass(I)$ be a maximal associated prime of $I$. Let $h$ be the height of $P$.
We first claim that $e \conv(\lattice{Q_{\subseteq P}}) \subseteq \stairs{Q_{\subseteq P}}$.  
Take $\abar \in \conv(\lattice{Q_{\subseteq P}})$ and express it as
\[ \abar = \sum_{i=1}^h \lambda_{P, i} \vbar_{P,i} + \sum_{j=0}^n c_{P,j} \ebar_j
\]
as in the proof of Lemma \ref{prop:approaches}. As $\sum_{i=1}^h \lambda_{P, i} = 1$ and each $\lambda_{P, i} \geq 0$, there exists an index $i_0$ with $\lambda_{P, i_0} \geq 1/h \geq 1/e$. Thus, $e\abar \geq e \lambda_{i_0} \vbar_{P,i_0} \geq \vbar_{P,i_0}$ where $\xbar^{\vbar_{P,i_0}} \in Q_{\subseteq P}$. Consequently, $e\abar \in \stairs{Q_{\subseteq P}}$. 

Thus, if we take a point $\abar \in \symbpoly$, then $\abar \in \conv(\lattice{Q_{\subseteq P}})$ for all $P \in \maxass(I)$ and hence $e\abar \in \bigcap_{P \in \maxass(I)} \stairs{Q_{\subseteq P}} = \stairs I$. That is, $e\symbpoly \subseteq \stairs I$.  If we take the $r$-fold Minkowski sum of both sides of the containment $e\symbpoly \subseteq \stairs I$ we get $er\symbpoly \subseteq r\stairs I $.  The results follows from the observation that $r \stairs I = \stairs{I^r}$.
\end{proof}
\end{prop}

The previous proposition and Theorem \ref{thm:symbcontainment} together prove the containment $I^{(nr)} \subseteq I^r$ of Ein-Lazarsfeld-Smith \cite{refELS} and Hochster-Huneke \cite{refHoHu} for monomial ideals. Also, $er\symbpoly \subseteq \stairs{I^r}$ implies $\alpha(\symbpoly) \geq \alpha(I)/e$, which in turn implies the Waldschmidt \cite{refW} and Skoda \cite{refSk} bound on $\alpha(I^{(m)})/m$.

The next theorem uses a simple degree argument to show that $\alpha(\symbpoly)$ controls which powers $s$ of the maximal ideal satisfy the containment $I^{(m)} \subseteq \maxideal^s I^r$ for $m$ sufficiently large.

For a homogeneous ideal $I$, let $\beta(I)$ be largest degree of a minimal generator of $I$.

\begin{thm} \label{thm:alphaslope}
Suppose $\symbpoly$ is the symbolic polyhedron of a monomial ideal $I$ with big-height $e$. For all integers $r \geq 0$ and $m \geq \max(er, \beta(I^r)/\alpha(\symbpoly))$,
\[ I^{(m)} \subseteq \maxideal^{\lceil \alpha(\symbpoly)m \rceil - \beta(I^r) } I^r.
\]

\begin{proof}
Take an arbitrary monomial $\xbar^\abar \in I^{(m)}$. Since $\lattice{I^{(m)}} \subseteq m \symbpoly \subseteq er \symbpoly$, we have $\abar \in I^r$ by Proposition \ref{prop:erQ}. Let $\xbar^\bbar$ be a minimal generator of $I^r$ that divides $\xbar^\abar$.  By Corollary \ref{cor:alphaq}, $\deg \xbar^\abar \geq \alpha(I^{(m)}) \geq m \alpha(\symbpoly)$. As $\deg \xbar^\bbar \leq \beta(I^r)$, we have $\deg \xbar^{\abar - \bbar} \geq m \alpha(\symbpoly) - \beta(I^r)$ and hence $\xbar^\abar \in \maxideal^{\lceil \alpha(\symbpoly) m \rceil - \beta(I^r)} I^r$. The desired containment follows.
\end{proof}
\end{thm}


Chudnovsky \cite{refCh} conjectured that if $I \subseteq \field[\mathbb P^n]$ is a radical ideal of points then 
\[ \alpha(I^{(m)})/m \geq (\alpha(I) + n - 1)/n \quad\text{for all $m$}.
\]
Since $\gamma(I) = \lim_{m\to \infty} \alpha(I^{(m)})/m$, we see that  
\[  \gamma(I) \geq (\alpha(I) + n - 1)/n \quad\text{for all $m$}
\]
follows from Chudnovsky's conjecture. 
As $\alpha(I^{(m)})/m \geq \gamma(I)$, for each $m$, they are, in fact, equivalent.

Our computational evidence suggests a similar conjecture for monomial ideals and their symbolic polyhedra, wherein $e$ replaces $n$:

\begin{conj} \label{conj:alpha}
If $I$ is a monomial ideal with big-height $e$ and $\symbpoly$ is the symbolic polyhedron of $I$, then $\alpha(\symbpoly) \geq (\alpha(I) + e - 1)/e$.
\end{conj}

The substitution of $e$ with $n$ makes for a stronger conjecture since $e \leq n$ except in the one case where $e=n+1$ and consequently $I^m = I^{(m)}$ (cf. Remark \ref{rem:nplusone}).


\begin{prop} \label{prop:equigen}
Suppose $I$ is an equigenerated ideal with big-height $e$. If Conjecture \ref{conj:alpha} holds for $I$ then $I^{(er)} \subseteq \maxideal^{ (e-1) r} I^r$ for all $r$.
\begin{proof}
We will use Theorem \ref{thm:alphaslope} with $m=er$. First, $m = er \geq r (\alpha(I) + e - 1)/\alpha(Q)$ 
since Conjecture \ref{conj:alpha} holds for $I$. Furthermore, $r (\alpha(I) + e - 1)/\alpha(Q) \geq r \alpha(I)/\alpha(Q) = \beta(I^r)/\alpha(Q)$ as $I$ is equigenerated. Thus, the condition $m \geq \max(er, \beta(I^r)/\alpha(Q))$ of Theorem \ref{thm:alphaslope} is satisfied.

Since $\lceil \alpha(\symbpoly) m \rceil - \beta(I^r)  = \lceil \alpha(\symbpoly) er \rceil - \beta(I^r) \geq (\alpha(I) + e - 1)r - \alpha(I)r = (e - 1) r$, the desired containment follows from Theorem \ref{thm:alphaslope}.
\end{proof}
\end{prop}

We now examine certain situations where the Conjecture \ref{conj:alpha} holds.
 
\begin{prop} \label{prop:ifchudholds}
Suppose $I$ is a monomial ideal and $\symbpoly$ is its symbolic polyhedron. If $\alpha(\symbpoly) = \alpha(I)$ then Conjecture \ref{conj:alpha} holds for $I$. Furthermore, if $\beta(I) \leq e \alpha(I)$ then $I^{(er)}  \subseteq \maxideal^{ (e-1) r} I^r$ for all $r$.
\begin{proof}
First $\alpha(\symbpoly) \geq 1$ since each $Q_{\subseteq P}$ is generated in degree at least one and hence $\symbpoly$ is contained in the half-space $x_0 + \cdots + x_n \geq 1$.

As $\alpha(\symbpoly) \geq 1$, we have $(e-1) \alpha(\symbpoly) \geq e-1$ and hence $e \alpha(\symbpoly) \geq \alpha(\symbpoly) + e - 1$. So,
\[ 
	\alpha(\symbpoly) 
	\geq \frac{\alpha(\symbpoly) + e -1}{e} 
	= \frac{\alpha(I) + e - 1}{e}.
\]

In order to show that $I^{(er)}  \subseteq \maxideal^{ (e-1) r} I^r$, it will suffice to show that $er \geq \max(er, \beta(I^r)/\alpha(\symbpoly))$ and apply Theorem \ref{thm:alphaslope}. This holds for all $r$ since $\beta(I^r)/ \alpha(\symbpoly) \leq r \beta(I) / \alpha(\symbpoly) =  r \beta(I)/ \alpha(I) \leq  er$.
\end{proof}
\end{prop}

The next result shows that certain ideals satisfy a variant of Conjecture \ref{conj:alpha} where $e$ is replaced with $n$. Recall that a monomial ideal $I$ is integrally closed if and only if $\conv(\lattice I) \cap \mathbb N^{n+1} = \lattice I$. That is, any lattice point $\abar \in \conv(\lattice I) \cap \mathbb N^{n+1} $ corresponds to a monomial $\xbar^\abar \in I$.

\begin{prop} \label{prop:intclosed}
Let $I$ be a monomial ideal and suppose that $Q_{\subseteq \p} = R \cap I R_\p$ is integrally closed for every $\p \in \maxass(I)$. 
If $n \geq 3$ and $\alpha(I) \geq n+4$ then $\alpha(\symbpoly) \geq (\alpha(I) + n - 1)/n$.
Also, if $n=2$ and $\alpha(I) \geq n + 6 = 8$ then $\alpha(\symbpoly) \geq (\alpha(I) + n - 1)/n = (\alpha(I) + 1)/2$.
\begin{proof}
Let $\symbpoly$ be the symbolic polyhedron of $I$.
Take $\abar = (a_0, \ldots, a_n) \in \symbpoly$ with $\abar \cdot \onebar = \alpha(\symbpoly)$ and let $\bbar = ( \lceil a_0 \rceil, \ldots, \lceil a_n \rceil )$. Now $\xbar^\bbar$ is in $I$ since, for each $P \in \maxass(I)$, $\bbar \in \conv(\lattice{Q_{\subseteq P}}) \cap \mathbb N^{n+1} = \lattice{Q_{\subseteq P}}$.
Thus $\alpha(I) \leq \bbar \cdot \onebar \leq \abar \cdot \onebar + n+1 = \alpha(\symbpoly) + n + 1$.

Since $\alpha(I) \geq n+4$ and $n \geq 3$, we have $\alpha(I) \geq n + 3 + \frac{2}{n-1}$, so $(n-1)\alpha(I) \geq n^2 +2n -1$.
Thus $n\alpha(I) \geq \alpha(I) + n^2 + 2n -1$ and hence $n\alpha(I) - n^2 -n \geq \alpha(I) + n - 1$. Dividing by $n$ gives,
$\alpha(\symbpoly) \geq \alpha(I) - n - 1 \geq \frac{\alpha(I) + n - 1}{n}$.

A similar argument works when $n=2$ and $\alpha(I) \geq n + 6$.
\end{proof}
\end{prop}

\bibliographystyle{amsplain}
\bibliography{sympowermonideal}

\end{document}